\documentclass[letterpaper, 10pt, conference]{ieeeconf}

\IEEEoverridecommandlockouts
\usepackage{amsmath,amssymb,amsfonts,mathtools}
\usepackage{tikz}
\usetikzlibrary{arrows.meta,positioning}
\mathtoolsset{showonlyrefs=true}
\usepackage[ruled,vlined]{algorithm2e}
\usepackage{graphicx}

\newcommand{\R}{\mathbb{R}}
\newcommand{\E}{\mathbb{E}}
\newcommand{\ud}{\,\mathrm{d}}
\newcommand{\BAR}[1]{\overline{#1}}
\newcommand{\EPS}{\varepsilon}
\newcommand{\transpose}{^{\operatorname{T}}}

\newcommand{\Tmin}{T_{\mathrm{min}}}
\newcommand{\Tcov}{T_{\mathrm{cov}}}

\newcommand{\smax}{s_{\mathrm{max}}}

\usepackage{amsthm}

\newtheorem{theorem}{Theorem}
\newtheorem{lemma}{Lemma}
\newtheorem{prop}{Proposition}
\newtheorem{remark}{Remark}
\newtheorem{assu}{Assumption}
\newtheorem{cor}{Corollary}

\usepackage{xcolor}
\usepackage{cite}
\usepackage[colorlinks=true, allcolors=blue]{hyperref}

\title{\LARGE \bf The Price of Covertness: Dual-Control Navigation in Uncertain Flows under Adversarial Sensing}

\author{Ruimeng Hu$^{1,2}$, Botao Jin$^{2}$, and Xu Yang$^{1}$
\thanks{This work was partially supported by the ONR grant \#N00014-24-1-2432, the Simons Foundation (MP-TSM-00002783), and the NSF grant DMS-2420988.}
\thanks{$^{1}$Department of Mathematics, University of California, Santa Barbara, CA 93106, USA.
Email: \texttt{\{rhu, xy6\}@ucsb.edu}.}
\thanks{$^{2}$Department of Statistics and Applied Probability, University of California, Santa Barbara, CA 93106, USA.
Email: \texttt{b\_jin@ucsb.edu}.}
}

\begin{document}

\maketitle
\thispagestyle{empty}
\pagestyle{empty}

\begin{abstract}

Covert navigation in an uncertain flow couples motion planning with information acquisition. A vehicle must estimate the local flow to navigate, while estimation error induces corrective maneuvers that increase statistical distinguishability, or leakage, and hence the vehicle's detectability. We model this coupling by augmenting position with the estimation-error variance, whose evolution depends on the route through the local information rate. A small-error expansion shows that estimation uncertainty contributes a leading-order correction to the expected detectability rate. The resulting route-planning problem is characterized by a first-order Hamilton--Jacobi--Bellman (HJB) equation. The deadline-indexed value determines the covert-time frontier, while an exact sensitivity formula with respect to sensing quality shows that information acquired earlier along the route can reduce leakage incurred later, whereas information acquired later cannot reduce leakage already incurred. We then formulate a zero-sum sensing-allocation game, show that the leakage rate is convex in the sensing allocation, and prove that the defender may restrict attention to randomization over extreme single-site allocations. The resulting equilibrium is computed by column generation, with vehicle best responses obtained from the HJB equation. Numerical experiments illustrate the learning detour, the deadline--leakage tradeoff, the spatial-ordering effect, and the benefit of randomized sensing allocations.

\end{abstract}

\section{Introduction}\label{sec:I}

Covert path planning asks how an agent can reach a goal while limiting the information available to an observer. Existing formulations typically treat the environment governing the agent's motion as known during planning, as in work on covert and deceptive planning~\cite{Ornik2018,Savas2022} and information-theoretic privacy in control~\cite{Farokhi2019,Nekouei2019}. This assumption is restrictive for navigation in flow fields, where a vehicle may have only a forecast of the ambient current and must refine that forecast from local onboard measurements. Recent guidance methods for unmanned underwater vehicles based on reinforcement learning~\cite{Greeley2023} explicitly operate under such local sensing, achieving near-time-optimal transits while observing only the current and its gradient near the vehicle.

In our setting, the moving agent is a vehicle navigating an uncertain flow field, while the observer is a defender seeking to detect its presence. Environmental uncertainty affects covertness through the corrective maneuvers induced by estimation error. The magnitude of these maneuvers depends on the vehicle's estimation error, and the resulting corrections contribute to the statistical evidence available to the defender. At the same time, the rate at which the vehicle reduces its estimation error depends on the information encountered along its route. The vehicle may therefore accept additional transit time to pass through an informative region, reducing the leakage incurred later. Route choice thus affects both motion and information acquisition, giving rise to the dual-control structure studied in this paper.

\smallskip
\noindent\textbf{Contributions.}
The novelty of this work lies in how environmental learning enters covert navigation. First, unlike covert path-planning formulations that treat the environment as known, we incorporate the vehicle's uncertainty about the flow into the routing problem and connect it to detectability through the corrective maneuvers induced by estimation error. This leads to a reduced information-state model in which the route governs both motion and the evolution of the estimation-error variance (Section~\ref{sec:model}).

Second, we formulate the resulting vehicle problem as a deterministic dual-control problem and characterize its value through a first-order Hamilton--Jacobi--Bellman (HJB) equation (Section~\ref{sec:red}). We further derive a sensitivity relation that quantifies the value of environmental information for covert navigation. In contrast to informative path planning, where information acquisition supports navigation or estimation, here its value also arises from its effect on future detectability: information acquired earlier can reduce the leakage generated by subsequent corrective maneuvers.

Third, we make the defender's sensing configuration strategic rather than fixed (Section~\ref{sec:game}). We show that the resulting allocation favors randomization over extreme allocations and leads to a finite-support equilibrium. These structural results are combined with dynamic programming and column generation to obtain a computational method (Section~\ref{sec:computation}) for the coupled vehicle--defender zero-sum game.

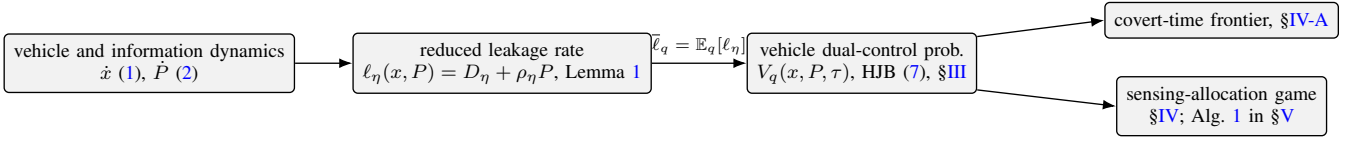
\begin{figure*}[t]
\centering
\resizebox{\textwidth}{!}{
\begin{tikzpicture}[
  font=\footnotesize,
  box/.style={draw, rounded corners=2pt, align=center, inner sep=4pt, fill=black!5},
  >={Latex[length=1.9mm]}, line width=0.55pt]
\node[box] (dyn) at (0,0) {vehicle and information dynamics\\ $\dot x$ \eqref{eqn:dyn}, $\dot P$ \eqref{eqn:riccati}};
\node[box] (leak) at (5.1,0) {reduced leakage rate\\ $\ell_\eta(x,P)=D_\eta+\rho_\eta P$, Lemma~\ref{lem:leak}};
\node[box] (hjb) at (10.3,0) {vehicle dual-control prob. \\ $V_q(x, P, \tau)$, HJB \eqref{eqn:hjb}, \S\ref{sec:red}};
\node[box] (fro) at (15.5,0.62) {covert-time frontier,  \S\ref{sec:frontier}};
\node[box] (game) at (15.5,-0.62) {sensing-allocation game \\ \S\ref{sec:game}; Alg.~\ref{alg:do} in \S\ref{sec:computation}};
\draw[->] (dyn) -- (leak);
\draw[->] (leak) -- (hjb) node[midway,above,font=\scriptsize]{$\BAR\ell_q=\E_q[\ell_\eta]$};
\draw[->] (hjb.13) -- (fro.west);
\draw[->] (hjb.-13) -- (game.west);
\end{tikzpicture}}
\caption{Structure of the paper. Motion and information dynamics determine the leakage rate and dual-control HJB, whose solution yields the covert-time frontier and best responses for the sensing game.}
\label{fig:roadmap}
\end{figure*}

\smallskip
\noindent\textbf{Relation to prior work.} The coupling between control and route-dependent information acquisition is closely related to the dual effect in stochastic control \cite{Feldbaum1960,BarShalomTse1974} and to informative path planning \cite{HollingerSukhatme2014}, but here information acquisition is also valued for its effect on future detectability.

More directly, this work builds on the Hamilton--Jacobi navigation framework of \cite{HuYang2026paper1}, where an observer infers a navigation objective from an observed trajectory. Here the arrival-time function from that framework reappears in the deadline formulation, while the information structure is different: the vehicle learns an uncertain environment as the defender accumulates evidence about its presence. The sequential-testing formulations of \cite{zhou2025integrating,zhou2025adversarial} consider a complementary adversarial-inference setting. The defender's sensing-allocation problem is also related to security games \cite{Tambe2011}, while related dynamic-game models study information revelation \cite{ralston2026information} and deception \cite{kim2026deception}.

\section{Model and Information Structure}\label{sec:model}

\subsection{Flow model and vehicle dynamics}

Let $\Omega\subset\R^d$, $d\in\{2,3\}$, be a bounded domain with target set $\mathcal T\subset\Omega$. Let the true current field be $v^{\mathrm{true}}(x) = \BAR v(x) + \delta v(x)$, where $\BAR v$ denotes the nominal forecast field and $\delta v$ the forecast error. The field $\BAR v$ is known to both parties, whereas $\delta v$ is unknown. The vehicle moves with maximum speed $\smax$ relative to the water. Under the reduced model introduced in Assumption~\ref{assu:ce}, its nominal dynamics are
\begin{equation}\label{eqn:dyn}
\dot x_s = \smax\, u_s + \BAR v(x_s),\; u_s\in U:=\{u\in\R^d: |u|\le 1\},
\end{equation}
being used for route planning. The effect of $\delta v$ on corrective maneuvers and detectability is modeled in Sections~\ref{sec:II-B} and \ref{sec:II-C}. The vehicle must reach $\mathcal T$.

Covertness concerns the defender's ability to \emph{detect} the vehicle against a background detection process, not to infer a hidden destination: the defender tests the null $H_0$ (no vehicle; detections arise from background clutter at rate $\lambda_0$) against the alternative $H_1$ (vehicle present, moving under \eqref{eqn:dyn}; detection rate $\lambda_1$). The vehicle seeks to reach $\mathcal T$ while limiting the accumulated statistical distinguishability between $H_1$ and $H_0$.

\subsection{The vehicle's information state and uncertainty dynamics}\label{sec:II-B}

The vehicle observes the flow only in a neighborhood of its position,  and uses these local observations to update its estimate of the unknown $\delta v$. To obtain a finite-dimensional information state, we consider a single-mode representation $$\delta v(x)=\phi(x)\theta,$$ 
where $\theta\in\R$ is an unknown amplitude and $\phi: \Omega \to \R^d$ is a known profile. Let $\hat \theta_s$ denote the vehicle's estimate of $\theta$ at time $s$, $e_s :=\theta-\hat\theta_s$ the estimation error, and $P_s:= \mathrm{Var}(e_s)$ its variance.  We take the variance to evolve according to a Riccati law
\begin{equation}\label{eqn:riccati}
\dot P_s = g(x_s,P_s;\kappa) := -\kappa\,\omega(x_s)\,P_s^2,
\end{equation}
where $\kappa \geq 0$ parameterizes sensing quality with $\kappa=0$ corresponding to no learning and large $\kappa$ to faster information acquisition. $\omega(x)\in[0,1]$ describes the local informativeness of the flow at $x$, and denotes the normalized Fisher-information rate of the onboard measurement about $\theta$.

We adopt \eqref{eqn:riccati} as a reduced deterministic information-state model, motivated by the covariance dynamics of continuous-time parameter estimation and Kalman--Bucy filtering \cite{KalmanBucy1961}. Equivalently, $ \frac{d}{ds}P_s^{-1}=\kappa\omega(x_s)$, so that the estimation precision accumulates at the local information rate $\kappa\omega(x_s)$. Rather than carrying the full stochastic estimation process into the planning problem, we retain only $P_s$ as the information state. Since the local information rate $\omega(x_s)$ depends on position, the evolution of $P_s$ depends on the route taken by the vehicle.

At each time $s$, we write the estimation error as $e_s=\sqrt{P_s}\,\xi_s$ where $\E[\xi_s]=0$, $\E[\xi_s^2]=1$ and  $\xi_s$ has a time-independent marginal distribution with bounded support. Thus $P_s$ determines the scale of the estimation error in the reduced model.

\begin{assu}\label{assu:ce}
Assume the forecast error is small relative to $\smax$, $|\delta v|/\smax=O(\epsilon)$ with $\epsilon\ll1$. The vehicle selects a nominal control $u_s\in U$ based on the forecast field $\BAR v$, with nominal trajectory \(x_s\) governed by \eqref{eqn:dyn}. The estimation error $e_s$ induces a corrective maneuver $r_s$, modeled by the linear response $r_s = K(x_s)e_s$, where \(K: \Omega \to \mathbb R^d\) is a prescribed correction gain. The effect of $r_s$ on the route-planning dynamics \eqref{eqn:dyn} is neglected, while its effect on the defender’s point-process detection model is retained; see Section~\ref{sec:II-C}.
\end{assu}

The term $r_s$ is not an independent control variable, but the first-order response of the low-level controller to the estimation error $e_s$. Although $r_s = K(x_s)\sqrt{P_s}\xi_s$ is neglected from the nominal planning dynamics~\eqref{eqn:dyn}, its effect is retained in the detection model of Section~\ref{sec:II-C}. Consequently, $P_s$ remains relevant to the routing problem and is retained as part of the reduced state.

Since $\kappa, \omega(x) \geq 0$, \eqref{eqn:riccati} implies $\dot P_s\le0$, while $P = 0$ is an equilibrium. Hence, for initial variance $P_0$,  $P \in [0, P_0]$. The reduced state is therefore $(x,P)\in\Omega\times[0,P_0]$ with dynamics
\begin{equation}\label{eqn:augdyn}
\dot x_s = f(x_s,u_s):=\smax u_s + \BAR v(x_s),\; \dot P_s = g(x_s,P_s;\kappa).
\end{equation}

\subsection{The defender's sensing allocation and detection model} \label{sec:II-C}

The defender allocates a total sensing budget $\Pi >0$ over $S$ candidate sites,
\begin{equation}\label{eqn:simplex}
\eta\in\mathcal A:=\Big\{\eta\in\R_+^S:\textstyle\sum_{j=1}^S\eta_j = \Pi\Big\}.
\end{equation}
The allocation $\eta$ determines the defender's spatial sensing capability. We model the resulting detection events as a point process: under hypothesis $H_i$, $i\in\{0,1\}$, let \(\lambda_i(x,r;\eta)\) denote the event intensity when the vehicle is at \(x\), undergoes corrective maneuver \(r\), and the sensing allocation is \(\eta\).

\begin{assu}\label{assu:affine}
For $i = 0, 1$, the intensity \(\lambda_i\) is continuous in $(x,r,\eta)$, affine in \(\eta\), and satisfies $0<\underline\lambda\le\lambda_i\le\BAR\lambda<\infty$. 
\end{assu}

To quantify the detectability of the vehicle's presence, we use the Kullback--Leibler divergence rate between the point-process models under $H_1$ and $H_0$. For Poisson intensities $a, b >0$, the rate is 
\begin{equation}\label{eqn:idiv}
\mathsf D(a,b) := a\log\tfrac ab - a + b.
\end{equation}
The positive intensity bounds in Assumption~\ref{assu:affine} ensure that \(\mathsf D\) is smooth with bounded derivatives on the admissible range $[\underline \lambda, \BAR \lambda]$, as required in Proposition~\ref{prop:wellposed}. Then, for a realized corrective maneuver \(r\), the instantaneous detectability rate is $\mathsf D(\lambda_1(x, r;\eta), \lambda_0(x, r;\eta))$. From Section~\ref{sec:II-B}, $r = K(x) e = K(x) \sqrt{P} \xi$. Hence the instantaneous detectability rate is random even for fixed $x, P, \eta$. We therefore define the exact expected detectability rate:
$$
\ell^{\mathrm{ex}}_\eta(x,P) := \E_\xi\big[\mathsf D\big(\lambda_1(x,K(x)\sqrt P\xi;\eta),\lambda_0(x, K(x)\sqrt P \xi; \eta)\big)\big].
$$

\begin{lemma}\label{lem:leak}
Assume, for each fixed $x$ and $\eta$,  $\lambda_i(x,\cdot\,;\eta)$ are affine in $r$ 
with the bounds in Assumption~\ref{assu:affine}. Then 
\begin{equation}\label{eqn:leak}
\ell^{\mathrm{ex}}_\eta(x,P) = \underbrace{\mathsf D\big(\lambda_1(x,0;\eta),\lambda_0(x, 0; \eta)\big)}_{=:D_\eta(x)}
+ \rho_\eta(x)\,P + o(P),
\end{equation}
as $P \to 0$, where $\rho_\eta(x):=K(x)\transpose\Theta_\eta K(x)$ and $\Theta_\eta(x) := \tfrac12\nabla^2_{rr}\mathsf D\big(\lambda_1(x,r;\eta),\lambda_0(x,r;\eta)\big)\big|_{r=0}$. Moreover, $\Theta_\eta(x) \succeq 0$ and hence $\rho_\eta(x) \geq 0$.
\end{lemma}
\begin{proof}
Define  $F(r) := \mathsf D\left( \lambda_1(x,r;\eta), \lambda_0(x,r;\eta) \right)$ for fixed \(x\) and \(\eta\). By the positive intensity bounds in Assumption~\ref{assu:affine}, \(\mathsf D\) is smooth on the relevant range. Since both intensities are affine in \(r\), \(F\) is \(C^2\) near \(r=0\). Taylor expansion gives
$$ F(r) = F(0) +\nabla F(0)^\top r +\frac12 r^\top\nabla^2F(0)r +o(|r|^2). $$
Substituting $r=K(x)\sqrt P\,\xi$ and using \(\mathbb E[\xi]=0\) and \(\mathbb E[\xi^2]=1\), we obtain
$$ \mathbb E_\xi[F(r)] = F(0) + \frac P2 K(x)^\top\nabla^2F(0)K(x) +o(P). $$
The bounded support of \(\xi\) ensures that the expected Taylor remainder is \(o(P)\). Hence $\ell_\eta^{\rm ex}(x,P) = D_\eta(x)+\rho_\eta(x)P+o(P). $ It remains to show nonnegativity. For \(a,b>0\),
$$ \nabla^2 \mathsf D(a,b) = \begin{pmatrix} \frac1a & -\frac1b\\[1mm] -\frac1b & \frac{a}{b^2} \end{pmatrix} = \frac1a \begin{pmatrix} 1\\[1mm]-a/b \end{pmatrix} \begin{pmatrix} 1&-a/b \end{pmatrix} \succeq0. $$
Thus \(\mathsf D\) is jointly convex in its two arguments. Since  $r\mapsto \bigl( \lambda_1(x,r;\eta), \lambda_0(x,r;\eta) \bigr)$ is affine, \(F\) is convex in \(r\). Therefore  $\Theta_\eta(x) = \frac12\nabla^2F(0) \succeq0, $ and consequently \(\rho_\eta(x)\ge0\).

\end{proof}

Thus \(D_\eta(x)\) is the nominal detectability rate corresponding to zero corrective maneuver, while \(\rho_\eta(x)P\) is the leading-order increase in expected detectability induced by estimation error. Thus, to leading order, increasing $P$ can only increase, rather than decrease,  $\ell^{\mathrm{ex}}_\eta(x,P)$ through its effect on the scale of $r$.

For the remainder of the paper, we specialize to a site-based detection model. For each site $j$, let $S_j: \Omega \to \R_+$ denote its spatial detection profile. Under allocation $\eta$, the aggregate detection profile is $S_\eta(x):=\sum_{j}\eta_jS_j(x)$. We then model the intensities as:
\begin{equation}\label{eqn:model}
\lambda_0=\lambda_{\mathrm{bg}},\quad
\lambda_1(x,r;\eta)=\lambda_{\mathrm{bg}}+a_1+S_\eta(x)\big(1+b\transpose r\big),
\end{equation}
where $\lambda_{\mathrm{bg}}>0$ is the background detection rate, $a_1\geq 0$ is the baseline contribution of the vehicle under $H_1$, and $b \in \R^d$ determines how the corrective maneuver \(r\) modifies the intensity. We assume that the model parameters and the bounded support of $\xi$ are such that the intensities in \eqref{eqn:model} satisfy the bounds in Assumption~\ref{assu:affine} uniformly for $x \in \Omega$, $P \in [0, P_0]$ and $\eta \in \mathcal A$. Then Lemma~\ref{lem:leak} gives $\Theta_\eta(x)=\tfrac{S_\eta^2(x)}{2\lambda_1(x, 0 ;\eta)}bb\transpose$ and 
\begin{equation}\label{eqn:rho}
\rho_\eta(x)=k(x)\,\frac{S_\eta(x)^2}{\lambda_1(x,0;\eta)},\quad k(x)=\tfrac12\big(b\transpose K(x)\big)^2.
\end{equation}
We henceforth use  
\begin{equation}\label{eqn:ell}
\ell_\eta(x,P):=D_\eta(x)+\rho_\eta(x)P
\end{equation}
as the reduced detectability (leakage) rate and the running cost in the subsequent control and game formulations. Its time integral provides the corresponding leading-order approximation to the expected accumulated log-likelihood ratio under $H_1$ relative to $H_0$.

\begin{remark}[Convexity in sensing allocation]\label{rem:convexeta}
For the detection model~\eqref{eqn:model}, the map $\eta \mapsto \ell_\eta(x, P)$ is convex for each fixed $(x, P)$. Indeed, let  $c:=\lambda_{\mathrm{bg}}+a_1$ and $s:=S_\eta(x)$, then $\ell_\eta(x, P)=\mathsf D(c+s,\lambda_{\mathrm{bg}})+k(x)P\,s^2/(c+s)$. Since $s$ is affine in $\eta$, it suffices to verify convexity in $s$. 

Both terms in $\ell_\eta(x, P)$ are convex in $s$, with second derivatives $1/(c+s)$ and, up to the positive factor $2k(x)P$, $c^2/(c+s)^3$. Hence $\eta\mapsto\ell_\eta(x,P)$ is convex. The exact expected detectability rate $\ell^{\mathrm{ex}}_\eta(x,P)$ is also convex in $\eta$: for each realization of $\xi$, the intensity pair is affine in $\eta$, $\mathsf D$ is jointly convex, and expectation preserves convexity. This establishes the convexity condition required in Proposition~\ref{prop:extreme}.
\end{remark}

\subsection{Strategic information structure}\label{sec:info}

The defender commits, before the encounter, to a distribution $q\in\Delta(\mathcal A)$, where $\Delta(\mathcal A)$ denotes the set of probability distributions over sensing allocations. It then privately draws a single allocation $\eta\sim q$, which remains fixed throughout the encounter. The vehicle knows $q$, but \emph{not} the realized $\eta$, whereas the defender knows its realization.

Because the realized allocation is hidden, the vehicle must choose its route against the distribution $q$, rather than against a known allocation $\eta$. Accordingly the vehicle plans against the \emph{expected leakage rate}
\begin{equation}\label{eqn:ellbar}
\BAR\ell_q(x,P):=\E_{\eta\sim q}\big[\ell_\eta(x,P)\big].
\end{equation}
The averaging over $\eta$ is performed at the level of detectability rates $\ell_\eta$, rather than at the level of the point-process intensities $\lambda_i$. Indeed, because the KL divergence is jointly convex, averaging the intensities before evaluating the divergence generally yields a smaller value than averaging the resulting detectability rates.

\section{The Vehicle's Dual-Control Problem}\label{sec:red}

Fix a defender mixture $q$. For a control $u:[0,\tau]\to U$, let $\tau_{\mathcal T}(u):= \inf \{s\geq 0: x_s \in \mathcal T\}$ denote the first hitting time of the target $\mathcal T$. For a deadline $\tau>0$, define the admissible control set $\mathcal U(x,\tau) := \{u: [0, \tau] \to U: \tau_{\mathcal{T}}(u) \leq \tau\}$, where $x_s$ evolves according to \eqref{eqn:augdyn} with $x_0 = x$. The deadline-indexed value function is
\begin{equation}\label{eqn:value}
V_q(x,P,\tau):=\inf_{u\in\mathcal U(x,\tau)}\int_0^{\tau_{\mathcal T}(u)}\BAR\ell_q(x_s,P_s)\ud s .
\end{equation}
This is a Lagrange (exit-time) problem with a hard terminal requirement: the running cost is the accumulated leakage, while the deadline enters through the admissible set.

\subsection{Domain and boundary data}

\begin{assu}[Uniform controllability]\label{assu:control}
The field $\BAR v$ is Lipschitz on $\BAR\Omega$, and $\sup_{x\in\Omega}|\BAR v(x)| < \smax$.
\end{assu}

Let $\Tmin(x)$ be the minimum arrival time from $x$ to $\mathcal T$ under \eqref{eqn:augdyn}. It is the viscosity solution of the eikonal problem (cf. \cite{HuYang2026paper1}),
\begin{equation}\label{eqn:eikonal}
\smax|\nabla \Tmin| - \BAR v\cdot\nabla \Tmin = 1 \ \text{ in }\Omega\setminus\mathcal T,\quad \Tmin|_{\partial\mathcal T}=0 .
\end{equation}

\begin{prop}\label{prop:domain}
Under Assumption~\ref{assu:control}, $\mathcal U(x,\tau)\ne\emptyset$ if and only if $\tau\ge \Tmin(x)$; $\Tmin$ is Lipschitz on $\BAR\Omega$ and independent of $P$. Consequently $V_q$ is finite exactly on
\begin{equation}\label{eqn:domain}
\mathcal D:=\{(x,P,\tau): x\in\Omega,\ P\in[0,P_0],\ \tau\ge \Tmin(x)\},
\end{equation}
the epigraph of $\Tmin$. On the lower boundary $\tau=\Tmin(x)$ every admissible trajectory is time-optimal, so the boundary trace is the least accumulated leakage among time-optimal trajectories,
\begin{align}\label{eqn:bdry}
V_q(x,P,\Tmin(x)) &= G_q(x,P)\\
&\hspace{-20pt}:= \inf_{u \in \mathcal U(x, \Tmin(x))}\int_0^{\Tmin(x)}\BAR\ell_q(x_s,P_s)\ud s,
\end{align}
the infimum taken over controls whose trajectory reaches $\mathcal T$ in exactly $\Tmin(x)$.
\end{prop}
\begin{proof}
Under Assumption~\ref{assu:control}, finiteness and the Lipschitz property of $\Tmin$ follow from \cite[Sec.~III]{HuYang2026paper1} (see also \cite{BardiCapuzzo1997}). Independence of $P$ holds because the $x$-dynamics in \eqref{eqn:augdyn} do not involve $P$. When $\tau=\Tmin(x)$ the admissible set consists exactly of the time-optimal trajectories, over which $V_q$ minimizes the accumulated leakage, giving \eqref{eqn:bdry}.
\end{proof}

\begin{assu}\label{assu:trace}
For each defender mixture $q$, the boundary trace $G_q(x,P)$ defined in \eqref{eqn:bdry} is continuous on $\{\tau=\Tmin(x)\}$.
\end{assu}
Assumption~\ref{assu:trace} holds, for example, where the time-optimal trajectory is unique and varies continuously with $x$; at bifurcation points it is an additional regularity assumption.

\subsection{Dynamic programming and the HJB equation}

By dynamic programming, $V_q$ is a viscosity solution of
\begin{equation}\label{eqn:hjb}
\begin{split}
-\partial_\tau V_q + \inf_{u\in U}\big\{ f(x,u)\cdot\nabla_x V_q \big\} &+ g(x,P;\kappa)\,\partial_P V_q\\
&+ \BAR\ell_q(x,P) = 0
\end{split}
\end{equation}
on the interior of $\mathcal D$, with $V_q=0$ on $\mathcal T$ and $V_q = G_q$ on $\{\tau=\Tmin(x)\}$.

\begin{lemma}\label{lem:mono}
For every fixed mixture $q$ and $(x,\tau)$ with $\tau \geq \Tmin(x)$, the map $P\mapsto V_q(x,P,\tau)$ is nondecreasing. Wherever $V_q$ is differentiable, $\partial_PV_q\ge0$.
\end{lemma}

\begin{proof}
Fix $0\le P_1\le P_2\le P_0$ and an admissible control $u$, and let $P_s^{(i)}$ denote the solution of~\eqref{eqn:riccati} with initial variance $P_i$, under the same (x, $\kappa$). Since the $x$-dynamics do not depend on $P$, both solutions follow the same spatial trajectory. The scalar Riccati equation is order preserving, so $P_s^{(1)}\le P_s^{(2)}$ for all $s$. Each reduced rate $\ell_\eta$ is nondecreasing in $P$ (cf. Lemma~\ref{lem:leak}), hence the mixture-average $\BAR\ell_q(x,\cdot)$. The accumulated leakage under $u$ is no larger from $P_1$ than from $P_2$. Taking the infimum over the common admissible control set gives $V_q(x,P_1,\tau)\le V_q(x,P_2,\tau).$  
\end{proof}

Since $g(x, P; \kappa) = -\kappa \omega(x) P^2\le0$, the product $g\,\partial_PV_q \leq 0$ quantifies the reduction in continuation cost due to learning. Since the route determines \(\omega(x)\), it also controls this learning effect.

\begin{prop}[Well-posedness]\label{prop:wellposed}
Under Assumptions~\ref{assu:ce}--\ref{assu:trace}, suppose that $\Omega$ is viable under the admissible dynamics, $U$ is compact, $f$ and $g$ are bounded and Lipschitz in the state variables (uniform in $u$ for $f$), and $\BAR\ell_q$ is bounded and Lipschitz in $(x,P)$. 
Consider \eqref{eqn:hjb} on $\mathcal D$, with $V_q=0$ on $\mathcal T$, $V_q=G_q$ on $\{\tau=\Tmin(x)\}$, and the state-constraint viscosity boundary condition of \cite{Soner1986} on $\partial\Omega\times[0,P_0]$. 
Then $V_q$ is the unique bounded continuous viscosity solution. No boundary data are imposed at $P=0$ or $P=P_0$.
\end{prop}

\begin{proof}
Since $g(x,0;\kappa)=0$ and $g(x,P_0;\kappa)\le0$, the interval $[0,P_0]$ is invariant under \eqref{eqn:riccati}, so $(x,P)$ remains in the compact set $\Omega\times[0,P_0]$. Under the stated regularity, the Hamiltonian of \eqref{eqn:hjb} is continuous in the state variables and Lipschitz in the gradient variables, and satisfies the Crandall--Lions structure condition
\[
|H(x,P,p)-H(x',P',p)|\le C\,(1+|p|)\,|(x,P)-(x',P')|,
\]
which follows from the Lipschitz dependence of $f$, $g$, and $\BAR\ell_q$ on $(x,P)$. By Proposition~\ref{prop:domain} the initial surface $\{\tau=\Tmin(x)\}$ is Lipschitz, and its trace $G_q$ is continuous by Assumption~\ref{assu:trace}. Standard comparison for first-order state-constrained HJB equations then gives uniqueness in the class of bounded continuous viscosity solutions. Existence follows by verifying, through the dynamic programming principle, that $V_q$ is a viscosity solution.
\end{proof}

\subsection{The value of environmental information}\label{sec:vol}

We now fix the mission initial state \((x_0,P_0)\) and examine how the sensing quality \(\kappa\) affects leakage along a fixed route.

The dual effect arises because a route determines both the information acquired along the way and the uncertainty carried into later regions. Equation~\eqref{eqn:riccati} makes this dependence explicit:
\begin{equation}\label{eqn:Pclosed}
P_t = \frac{P_0}{1+\kappa P_0\int_0^t\omega(x_r)\ud r}.
\end{equation}

\begin{theorem}\label{thm:vol}
Fix a defender mixture $q$ and a route $\gamma = (x_s)_{0 \leq s \leq T_\gamma}$ from $x_0$ to $\mathcal T$, held fixed as $\kappa$ varies, where $T_\gamma$ is the first hitting time. Let $I_q(\gamma; \kappa) :=\int_0^{T_\gamma} \BAR \ell_q(x_s, P_s) \ud s$ be its accumulated leakage under sensing quality $\kappa$, and set $A_s:=\int_0^s\omega(x_r)\ud r$. Then, with $P$ given by \eqref{eqn:Pclosed}, for every $\kappa\ge0$,
\begin{equation}\label{eqn:volk}
\frac{dI_q(\gamma; \kappa)}{d\kappa}
= -P_0^2\!\int_0^{T_\gamma}\!\omega(x_r)\Big[\int_r^{T_\gamma}\!\frac{\partial_P \BAR\ell_q(x_s, P_s)}{(1+\kappa P_0A_s)^{2}}\ud s\Big]\ud r \le 0.
\end{equation}
\end{theorem}
\begin{proof}
By \eqref{eqn:Pclosed}, $\partial_\kappa P_s=-P_0^2A_s/(1+\kappa P_0A_s)^{2}$. Insert into $dI_q/d\kappa=\int_0^{T_\gamma}\partial_P \BAR\ell_q(x_s, P_s)\,\partial_\kappa P_s\,ds$ and exchange the order of integration over $\{(r,s):r\le s\}$ gives \eqref{eqn:volk}. Finally, $\partial_P\BAR\ell_q(x,P) = \mathbb E_{\eta\sim q}[\rho_\eta(x)] \ge 0$
by Lemma~\ref{eqn:leak}, so the derivative is nonpositive.
\end{proof}

Equation~\eqref{eqn:volk} shows a temporal ordering effect along the route: information acquired at time \(r\) can reduce leakage only at later times \(s\ge r\). Hence sensing is most valuable when informative regions are encountered before regions where leakage is sensitive to $P$, as measured by $\partial_P \BAR \ell_q$. If a $P$-sensitive region is encountered before the informative region, its leakage cannot benefit from information acquired later. This ordering effect is illustrated numerically in Section~\ref{sec:computation}. Since the admissible route set is independent of \(\kappa\), taking the infimum over routes in Theorem~\ref{thm:vol} shows that, for every fixed \(q\), \(V_q(x_0,P_0,\tau;\kappa)\) is nonincreasing in \(\kappa\).

\section{The Sensing-Allocation Game}\label{sec:game}

\subsection{Game value and covert-time frontier}\label{sec:frontier}

At a fixed deadline $\tau$, the defender chooses a sensing mixture $q\in\Delta(\mathcal A)$ to maximize the vehicle's expected leakage, then the vehicle chooses an admissible route $\gamma$ to minimize it. For a route $\gamma$ with first hitting time $T_\gamma$, let $I_\eta(\gamma) := \int_0^{T_\gamma} \ell_\eta(x_s,P_s)\ud s$ denote its accumulated leakage under sensing allocation $\eta$, and
let $\Gamma(\tau) := \{\gamma: T_\gamma \leq \tau\}$. Then, for a fixed defender mixture $q$, $
V_q(x_0,P_0,\tau) = \inf_{\gamma\in\Gamma(\tau)} \mathbb E_{\eta\sim q}[I_\eta(\gamma)]$, consistent with the vehicle's problem of Section~\ref{sec:red}. The resulting max-min value is
\[
W(\tau) := \sup_{q\in\Delta(\mathcal A)}V_q(x_0,P_0,\tau) = \sup_{q\in\Delta(\mathcal A)} \inf_{\gamma\in\Gamma(\tau)} \mathbb E_{\eta\sim q}[I_\eta(\gamma)].
\]

For a fixed sensing mixture $q$, define the covert-time
frontier by
\[
T_{\rm cov}^q(B) := \inf\{\tau:V_q(x_0,P_0,\tau)\le B\}.
\]
At the game level, define
\[
T_{\rm cov}^{\rm game}(B) := \inf\{\tau:W(\tau)\le B\},
\]
with the convention $\inf\emptyset=+\infty$. Thus
$T_{\rm cov}^{\rm game}(B)$ is the minimum deadline for which
the vehicle can keep its expected leakage below $B$ against every defender mixture. The \emph{price of covertness} for a fixed mixture $q$ is
\[
\Delta T_{\rm cov}^q(B) := T_{\rm cov}^q(B)-T_{\min}(x_0),
\]
and analogously at the game level.

We interpret $B$ as an upper bound on the expected log-likelihood evidence accumulated by the defender before arrival. Under the standard Wald large-threshold approximation, a detection threshold associated with a small target error probability $\EPS$ scales as $O(|\log\EPS|)$. Our formulation controls expected accumulated evidence; the corresponding realized stopping-time problem is studied in~\cite{zhou2025integrating,zhou2025adversarial}.

\begin{cor}\label{cor:mono}
For each fixed mixture $q$, $\Tcov^{q}(B;\kappa)$ is nonincreasing in $\kappa$, and so is the equilibrium frontier $\Tcov^{\mathrm{game}}(B;\kappa)$.

\end{cor}
\begin{proof}
We have proved \(V_q(x_0,P_0,\tau;\kappa)\) is nonincreasing in \(\kappa\). Taking the supremum over \(q\) preserves pointwise monotonicity, so \(W(\tau;\kappa)\) is also nonincreasing in \(\kappa\). Thresholding in \(\tau\) gives the result.
\end{proof}

For fixed \(q\), a single backward solution of the HJB yields \(V_q(x_0,P_0,\tau)\) for all deadlines, and hence the entire frontier \(T_{\rm cov}^q(B)\) by thresholding. At the game level, the maximizing mixture \(q^\star(\tau)\) may vary with \(\tau\), so the outer optimization must be repeated across deadlines.

\subsection{Minimax formulation}

To establish a minimax equality, we compactify the vehicle's route set using relaxed controls. Since $U$ is convex and the dynamics are affine in $u$, every relaxed control has a barycentric ordinary control generating the same state trajectory. Thus relaxation does not change the vehicle's value. We therefore retain the notation $\Gamma(\tau)$ for the relaxed-control closure of the admissible route set. A trajectory reaching the target before $\tau$ is extended to $[0,\tau]$ by an absorbing terminal state with zero running leakage.

The vehicle's mixed strategies are probability measures
$\sigma\in\Delta(\Gamma(\tau))$. Since the payoff is linear in
$\sigma$, for every fixed $q$, $ \inf_{\sigma\in\Delta(\Gamma(\tau))} \mathbb E_{\eta\sim q}\mathbb E_{\gamma\sim\sigma} [I_\eta(\gamma)] = V_q(x_0,P_0,\tau)$.

\begin{prop}[Minimax equality]
For every $\tau\ge T_{\min}(x_0)$,
\begin{align}
W(\tau) =  \sup_{q\in\Delta(\mathcal A)} \inf_{\sigma\in\Delta(\Gamma(\tau))} \mathbb E_{\eta\sim q}\mathbb E_{\gamma\sim\sigma} [I_\eta(\gamma)]\\
= \inf_{\sigma\in\Delta(\Gamma(\tau))} \sup_{q\in\Delta(\mathcal A)} \mathbb E_{\eta\sim q}\mathbb E_{\gamma\sim\sigma} [I_\eta(\gamma)].
\end{align}
Moreover, the extrema are attained.
\end{prop}

\begin{proof}
The  set $\mathcal A$ is compact, and $\Gamma(\tau)$ is compact
under the relaxed-control topology. Hence
$\Delta(\mathcal A)$ and $\Delta(\Gamma(\tau))$ are compact and convex
under weak convergence (by Prokhorov's theorem).  
Under the standing regularity assumptions,
$(\eta,\gamma)\mapsto I_\eta(\gamma)$ is bounded and continuous.
Therefore $(q,\sigma)\mapsto \mathbb E_{\eta\sim q}\mathbb E_{\gamma\sim\sigma} [I_\eta(\gamma)]$ is continuous and affine in each argument. Sion's minimax theorem~\cite{Sion1958} gives the stated equality, while compactness and continuity give attainment.
\end{proof}

Thus max-min value $W(\tau)$, introduced in Section~\ref{sec:frontier}, is also the minimax equilibrium value. The common deadline ensures that
$\Gamma(\tau)$ is independent of the defender's sensing
allocation, so both players optimize over fixed strategy sets.

\subsection{Extreme allocations and finite support}

Since $\mathcal A$ is the scaled simplex, its extreme points are $
\operatorname{ext}(\mathcal A) := \{\Pi e_1,\ldots,\Pi e_S\},$ corresponding to allocations that place the entire sensing budget at a single site. The convexity established in Remark~\ref{rem:convexeta} has two strategic consequences: randomization weakly improves upon the corresponding deterministic mean allocation, and it suffices to randomize over these extreme allocations.

\begin{prop}\label{prop:extreme}
For the detection model~\eqref{eqn:model}, fix $\tau \geq T_{\min}(x_0)$. For any $q\in\Delta(\mathcal A)$ with barycenter $\BAR\eta=\E_{\eta \sim q}[\eta]$,
\begin{equation}\label{eqn:jensen}
\inf_{\gamma \in \Gamma(\tau)} \E_{\eta\sim q}\big[I_\eta(\gamma)\big] \;\ge\; \inf_{\gamma \in \Gamma(\tau)} I_{\BAR\eta}(\gamma).
\end{equation}
Moreover, an optimal defender strategy may be chosen with support in $\operatorname{ext}(\mathcal A)$.
\end{prop}

\begin{proof}
For every fixed route $\gamma$, the convexity of $\eta\mapsto\ell_\eta(x,P)$ established in Remark~\ref{rem:convexeta} implies convexity of $\eta\mapsto I_\eta(\gamma)$. Hence Jensen's inequality gives $\mathbb E_{\eta\sim q}[I_\eta(\gamma)] \ge I_{\bar\eta}(\gamma)$.
Taking the infimum over $\gamma\in\Gamma(\tau)$ yields \eqref{eqn:jensen}.

Since $\mathcal A$ is the scaled simplex, every $\eta\in \mathcal A$ has the
representation $ \eta = \sum_{j=1}^S\frac{\eta_j}{\Pi}\,\Pi e_j$.
Therefore, again by convexity, $
I_\eta(\gamma)
\le
\sum_{j=1}^S\frac{\eta_j}{\Pi}
I_{\Pi e_j}(\gamma)$,
$\text{for every }\gamma$.
Thus each realized allocation $\eta$ can be replaced by the
corresponding mixture over the extreme allocations
$\{\Pi e_j\}_{j=1}^S$ without decreasing the defender's payoff
against any route. Applying this replacement to every allocation
in the support of $q$ yields a strategy supported on
$\operatorname{ext}(\mathcal A)$ that weakly dominates $q$. Hence an
optimal defender strategy may be chosen with support in
$\operatorname{ext}(\mathcal A)$.
\end{proof}

Thus a deterministic allocation that distributes the sensing budget across multiple sites is weakly dominated by a randomized strategy over single-site allocations with the same mean allocation. This conclusion follows from the convexity of the detectability model and does not imply that randomization is always necessary; a single extreme allocation may itself be optimal.

\begin{cor}\label{cor:support}
The vehicle has an equilibrium mixed strategy supported on at most $S$ routes.
\end{cor}

\begin{proof}
By Proposition~\ref{prop:extreme}, the defender may restrict to the $S$ extreme
allocations $\{\Pi e_j\}_{j=1}^S$. Each route $\gamma$ is therefore
identified with its payoff vector $z(\gamma):= \bigl(I_{\Pi e_1}(\gamma),\ldots,I_{\Pi e_S}(\gamma)\bigr)\in\mathbb R^S$. Let $(q^\star,\sigma^\star)$ be an equilibrium with value
$W(\tau)$. Every route used by $\sigma^\star$ is a best response
to $q^\star$, hence its payoff vector satisfies
$q^\star\!\cdot z(\gamma)=W(\tau)$. Thus these vectors lie in an
affine subspace of dimension at most $S-1$. By Carath\'eodory's
theorem, their mean payoff vector can be represented using at
most $S$ such vectors. The corresponding mixture is therefore
also an equilibrium vehicle strategy. 
\end{proof}

\begin{remark}\label{rem:kink}
For fixed $(x_0, P_0, \tau)$, the map $ q\mapsto V_q(x_0,P_0,\tau) = \inf_{\gamma\in\Gamma(\tau)} \mathbb E_{\eta\sim q}[I_\eta(\gamma)]$ is concave, being the infimum of functionals linear in $q$, and may be nonsmooth when several vehicle best-response routes are active. Such kinks correspond to indifference among active routes and can lead to equilibrium mixing by the vehicle.
\end{remark}

\section{Algorithm and Numerical Examples}\label{sec:computation}

By Proposition~\ref{prop:extreme}, the defender may restrict to the \(S\) extreme allocations \(\operatorname{ext}(\mathcal A)=\{\Pi e_j\}_{j=1}^S\), while the vehicle retains the route set $\Gamma(\tau)$. At a fixed deadline \(\tau\), this yields a zero-sum game with finitely many defender actions and an infinite set of vehicle routes. We solve this game by column generation over vehicle routes~\cite{GilmoreGomory1961,McMahan2003}.

Let $\Gamma_n=\{\gamma_0, \gamma_1,\ldots,\gamma_{n}\}\subset\Gamma(\tau)$ be the routes generated after $n$ iterations, and define the
restricted-game value
\[
\BAR W_n := \sup_{q\in\Delta(\operatorname{ext}(\mathcal A))} \inf_{\sigma\in\Delta(\Gamma_n)} \mathbb E_{\eta\sim q} \mathbb E_{\gamma\sim\sigma} [I_\eta(\gamma)].
\]
Since $\Gamma_n\subset\Gamma(\tau)$, $W(\tau)\le \BAR W_n$.

Let $(q_n,\sigma_n)$ be an equilibrium of the restricted game on $\operatorname{ext}(\mathcal A)\times\Gamma_n$. A best response to $q_n$ over full route set is
\[
\gamma_{n+1}
\in
\arg\min_{\gamma\in\Gamma(\tau)}
\mathbb E_{\eta\sim q_n}[I_\eta(\gamma)].
\]
By Section~\ref{sec:red}, this best response is obtained by solving the HJB equation~\eqref{eqn:hjb} against the mixture $q_n$ and backtracking an optimal trajectory. In column-generation terminology, this computation is the pricing oracle. Under exact pricing, define
\[
\underline W_n := \mathbb E_{\eta\sim q_n} [I_\eta(\gamma_{n+1})] = V_{q_n}(x_0,P_0,\tau).
\]
Since $W(\tau)=\sup_q V_q(x_0,P_0,\tau)$, $\underline W_n\le W(\tau)$. Therefore
$\underline W_n \le W(\tau) \le \BAR W_n$, and the gap $\BAR W_n-\underline W_n$ provides an a posteriori optimality certificate for the full game.

\begin{algorithm}[htbp]
\caption{Column generation with an HJB pricing oracle at deadline $\tau$}\label{alg:do}
\KwIn{$\operatorname{ext}(\mathcal A)$; deadline $\tau$; tolerance $\EPS$}
Initialize $\Gamma_0$ with a time-optimal route\;
\For{$n=0,1,2,\dots$}{
 Solve the restricted zero-sum game on $\operatorname{ext}(\mathcal A)\times\Gamma_n \to
(q_n,\sigma_n)$ and value $\BAR W_n$\;
Solve \eqref{eqn:hjb} against \(q_n\) and backtrack a best-response route \(\gamma_{n+1}\) at deadline $\tau$\;
Set $
    \underline W_n
    \leftarrow
    \mathbb E_{\eta\sim q_n}[I_\eta(\gamma_{n+1})]$\;
 \If{$\BAR W_n-\underline W_n\le\EPS$}{\Return $(q_n,\sigma_n)$\;}
 $\Gamma_{n+1}\gets\Gamma_n\cup\{\gamma_{n+1}\}$\;
}
\end{algorithm}

For a finite route discretization and exact pricing, column generation terminates after finitely many iterations. In our implementation, the HJB pricing oracle is itself
discretized. Hence the reported gap $\BAR W_n-\underline W_n$  certifies the resulting discretized game, while approximation of
the continuous problem additionally incurs HJB discretization error.

For $d=2$ and scalar $P$, the HJB is 3D in $(x_1,x_2,P)$ and is marched backward in $\tau$ using a monotone upwind or semi-Lagrangian scheme. Since $\dot P=-\kappa\omega(x)P^2\le0$, transport in the $P$-direction is one-sided, and the closed form~\eqref{eqn:Pclosed} is used to propagate $P$ along discrete trajectory segments.

\noindent\textbf{Numerical setup.} We consider the compactly supported vortex field of \cite{Greeley2023} $\BAR v(x)=d_r\max(m_r\|w\|+b_r,0)(-w_2,w_1)/\|w\|$ with $w=x-x_c$, on $\Omega=[-4,4]^2$. We take $\smax=1$,  $d_r=1$, $b_r=0.75$, $m_r=-0.6$, giving support radius $R=1.25$ and Assumption~\ref{assu:control} holds. The vortex center is $x_c=(1.5,-0.4)$, the vehicle starts from  $x_0=(2.8,2.4)$ and the target is $T=\{(-2.6,-2.2)\}$.  $\Omega$ is chosen large enough that computed optimal trajectories remain in its interior.

We set the local information rate to the normalized current speed $\omega(x)=\frac{|\BAR v(x)|}{\max_{y\in\Omega}|\BAR v(y)|}, \; \max_{y \in \Omega} |\bar v(y)| =b_r = 0.75$, so the most informative region is the vortex core. Because the vortex
center lies off the direct start-target corridor, acquiring information
there requires a deliberate detour. Unless otherwise stated, we use $\kappa=6$ and $P_0=1$. For the defender,  the total sensing budget is normalized to $\Pi=1$. Each candidate site has a Gaussian profile $S_j$ of unit amplitude and width $0.9$, and the intensities follow \eqref{eqn:model} with $\lambda_{\mathrm{bg}}=0.15$ and $a_1=0.08$. In \eqref{eqn:rho}, we take
$k(x) \equiv k_0=3$.

The arrival-time field $\Tmin$ is computed by the fast sweeping method \cite{Zhao2005}. Grid transit times use the compensated ground speed $\BAR v\cdot\hat e+(\smax^2-|\BAR v_\perp|^2)^{1/2}$ along each edge, consistent with the eikonal equation. On the $161\times161$ grid, the time-optimal graph transit is $7.218$, compared with $\Tmin(x_0)=7.020$. 
Minimum-leakage routes are computed by Dijkstra's algorithm on the augmented grid of $(x,P)$, with $P$ discretized into $70$ bins and propagated across each edge using~\eqref{eqn:Pclosed}. Deadline-indexed values by the backward-in-$\tau$ march of \eqref{eqn:hjb} with boundary data given in Proposition~\ref{prop:domain}. 

As a benchmark, we use a frozen-$P$ planner that plans with
$P\equiv P_0$, ignoring the effect of route choice on future
information acquisition. Its realized leakage is nevertheless
evaluated with $P$ evolving according to~\eqref{eqn:Pclosed}. Thus, both planners experience the same information dynamics during evaluation, and
the comparison isolates the value of planning for information
acquisition rather than the value of information acquisition itself.

\noindent\textbf{Numerical results.} 
Figure~\ref{fig:mechanism}(a,b) illustrates the learning-detour mechanism. A single sensor is located at the target, so some arrival exposure is unavoidable. The dual planner detours through the informative region beforehand, reducing $P$ before the main exposure. It accepts a $21\%$ longer transit time ($8.71$ versus $7.22$) while reducing realized leakage from $2.96$ to $1.80$, a $39\%$ reduction. Panel (c) illustrates the ordering effect of Theorem~\ref{thm:vol}. When detection occurs near departure, the main exposure precedes information acquisition and the benefit of learning is negligible; when detection occurs near arrival, information is acquired first and the leakage reduction is substantial. The improvement is largest at intermediate $\kappa$: learning is too slow for small $\kappa$, while for large $\kappa$ even the frozen-$P$ route acquires substantial information. This nonmonotonic dependence of the improvement on $\kappa$ does not contradict Theorem~\ref{thm:vol}, which concerns a fixed route, whereas panel (c) compares planners whose selected routes vary with $\kappa$.

\begin{figure*}[htbp]
\centering
\includegraphics[width=0.9\textwidth]{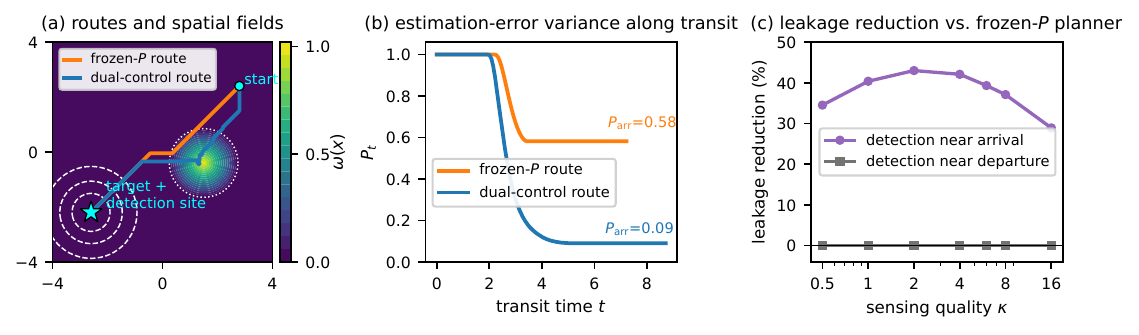}
\caption{Learning-detour mechanism. 
(a) The frozen-$P$ route (orange, nearly time-optimal) and dual-control route (blue);  color shows the local information
rate $\omega(x)$, dashed contours show level sets of $D_\eta(x)$,
and the dotted circle marks the vortex support. 
(b) Estimation-error variance $P_t$ along the two routes. (c) Leakage reduction of the dual-control planner relative to the frozen-$P$ planner as a
function of $\kappa$, for detection sites near
departure and near arrival. 
}
\label{fig:mechanism}
\end{figure*}

\begin{figure*}[t]
\centering
\includegraphics[width=0.9\textwidth]{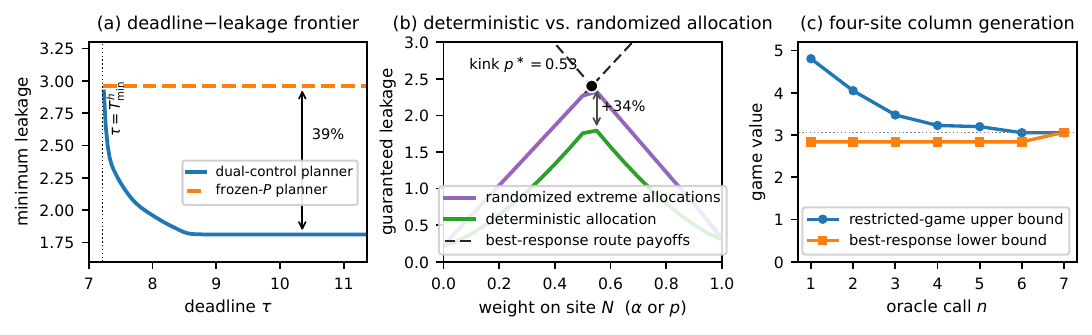}
\caption{(a) Deadline--leakage frontier for the single-site example. (b) Deterministic (green) allocation \emph{vs.} randomization (purple) for the two-site game at $\tau=1.35T_{\min}$ and $\kappa=0$. The kink at $p^\star=0.53$ occurs where two vehicle best-response routes have equal expected leakage. (c) Column generation for the four-site game. The restricted-game upper bound and the full-route best-response lower bound converge within seven oracle calls.}

\label{fig:game}
\end{figure*}

Figure~\ref{fig:game}(a) shows the deadline-indexed value for a fixed defender allocation. As the deadline increases, the dual-control
planner exploits the additional route flexibility to reduce
leakage, whereas the frozen-$P$ planner cannot attain the same
levels. Thresholding the value curve yields the covert-time
frontier $T_{\rm cov}^q$ of Section~\ref{sec:frontier}; for example, $B=2$ is
achieved at $\tau\approx7.9$, corresponding to a price of
covertness of about $0.88$ time units, or $12.5\%$ relative to
$T_{\min}(x_0)=7.020$.

Figure~\ref{fig:game}(b) exhibits the allocation effect of Proposition~\ref{prop:extreme} using two candidate sites $z_N = (2.2,0.8)$ and $z_S = (0.8,2.2)$. 
We take \(\kappa=0\) to isolate allocation convexity from learning.
We compare the deterministic split $\eta^\alpha=\alpha e_N+(1-\alpha)e_S$, $\alpha\in[0,1]$ with the randomized strategy $q^p=p\,\delta_{e_N}+(1-p)\delta_{e_S}$, $p\in[0,1]$. When $p=\alpha$, the two strategies have the same mean allocation. Nevertheless, randomization yields a larger guaranteed leakage, with maximum value $2.40$ at $p^\star=0.53$, compared with $1.79$ for the best deterministic split, an increase of approximately $34\%$. The kink at $p^\star$ occurs where two vehicle best-response routes
have equal expected leakage. 

Finally, Fig.~\ref{fig:game}(c) demonstrates the
column-generation procedure on a four-site game. In addition to the
two sites used in panel~(b), two candidate sites are placed near the
arrival region at $(-2.2,-0.7)$ and $(-0.7,-2.2)$. Algorithm~\ref{alg:do} reaches the prescribed
tolerance after seven HJB oracle calls. At the resulting equilibrium, the defender randomizes over all four sites, while the vehicle mixes over four of the seven generated
routes, with game value $3.06$. The restricted-game and pricing values agree to within $0.4\%$ of the game value.

\section{Conclusion}\label{sec:conclusion}

We developed a framework for covert navigation in an uncertain flow, where the vehicle's route affects both its motion and the information it acquires about the environment. By reducing the estimation process to the evolution of the estimation-error variance $P$, the model retains this route-dependent dual effect while yielding a deterministic control problem. For a fixed defender strategy and arrival deadline, the vehicle's dual-control problem is characterized by a first-order HJB equation in the state variables $(x, P)$ on the domain $\tau\ge T_{\min}(x)$, and thresholding its deadline-indexed value yields the covert-time frontier. A closed-form sensitivity formula further shows how the benefit of information depends on the temporal ordering of information acquisition and detection exposure. We then formulate the interaction with the defender as a zero-sum sensing-allocation game.  Convexity with respect to the sensing allocation allows the defender to restrict attention to randomization over extreme single-site allocations, while the resulting game can be solved by column generation using the HJB solve as a best-response oracle. The numerical examples illustrate the learning detour, the deadline--leakage tradeoff, and the advantage of randomization over deterministic budget splitting.

Several extensions remain. A natural next step is the noisy-sensing regime, in which the vehicle's estimate becomes a stochastic latent process and the defender must marginalize over that uncertainty, leading to coupled filtering problems. The strategic sensing game can also be combined with a tactical layer in which, conditional on a realized deployment, the defender adapts its sensing policy as its belief about the vehicle evolves. Incorporating declaration lead time as the mission objective and more realistic ocean-acoustic fields would further connect the framework to the sequential-testing setting of \cite{zhou2025integrating,zhou2025adversarial}.

\makeatletter
\renewenvironment{thebibliography}[1]{
  \section*{References}
  \footnotesize
  \list{\@biblabel{\@arabic\c@enumiv}}
       {\settowidth\labelwidth{\@biblabel{#1}}
        \leftmargin\labelwidth \advance\leftmargin\labelsep
        \itemsep -1pt plus .3pt \parsep \z@
        \usecounter{enumiv}\let\p@enumiv\@empty
        \renewcommand\theenumiv{\@arabic\c@enumiv}}
  \sloppy\clubpenalty4000\widowpenalty4000
  \sfcode`\.\@m}{\def\@noitemerr{\@latex@warning{Empty `thebibliography' environment}}\endlist}
\makeatother
\bibliographystyle{IEEEtran}
\bibliography{covert_refs}

@article{KalmanBucy1961,
  author  = {Kalman, Rudolf E. and Bucy, Richard S.},
  title   = {New Results in Linear Filtering and Prediction Theory},
  journal = {Journal of Basic Engineering},
  volume  = {83},
  number  = {1},
  pages   = {95--108},
  year    = {1961},
  doi     = {10.1115/1.3658902}
}

@article{HuYang2026paper1,
  title={Strategic Inference of Adversarial Navigation Objectives for Unmanned Underwater Vehicles},
  author={Hu, Ruimeng and Yang, Xu},
  journal={arXiv:2607.21945},
  year={2026}
}

@inproceedings{Greeley2023,
  title={Reinforcement learning for improved guidance and power management of unmanned underwater vehicles},
  author={Greeley, Brian and Brandman, Jeremy and Olson, Colin},
  booktitle={OCEANS 2023-MTS/IEEE US Gulf Coast},
  pages={1--10}
}

@article{Ornik2018,
  title={Exploiting partial observability for optimal deception},
  author={Karabag, Mustafa O and Ornik, Melkior and Topcu, Ufuk},
  journal={IEEE Transactions on Automatic Control},
  volume={68},
  number={7},
  pages={4443--4450},
  year={2022}
}

@inproceedings{Savas2022,
  title={Deceptive decision-making under uncertainty},
  author={Savas, Yagiz and Verginis, Christos K and Topcu, Ufuk},
  booktitle={Proceedings of the AAAI Conference on Artificial Intelligence},
  volume={36},
  number={5},
  pages={5332--5340},
  year={2022}
}

@article{Farokhi2019,
  title={Ensuring privacy with constrained additive noise by minimizing {F}isher information},
  author={Farokhi, Farhad and Sandberg, Henrik},
  journal={Automatica},
  volume={99},
  pages={275--288},
  year={2019},
  publisher={Elsevier}
}

@article{Nekouei2019,
  title={Information-theoretic approaches to privacy in estimation and control},
  author={Nekouei, Ehsan and Tanaka, Takashi and Skoglund, Mikael and Johansson, Karl H},
  journal={Annu. Rev. Control},
  volume={47},
  pages={412--422},
  year={2019},
  publisher={Elsevier}
}

@book{BardiCapuzzo1997,
  title={Optimal control and viscosity solutions of {H}amilton-{J}acobi-{B}ellman equations},
  author={Bardi, Martino and Dolcetta, Italo Capuzzo and others},
  volume={12},
  year={1997},
  publisher={Springer}
}

@article{Soner1986,
  title={Optimal control with state-space constraint {I}},
  author={Soner, Halil Mete},
  journal={SIAM J. Control Optim.},
  volume={24},
  number={3},
  pages={552--561},
  year={1986},
  publisher={SIAM}
}

@article{Zhao2005,
  title={A fast sweeping method for eikonal equations},
  author={Zhao, Hongkai},
  journal={Math. Comput.},
  volume={74},
  number={250},
  pages={603--627},
  year={2005}
}

@article{Sion1958,
  author = {Sion, Maurice}, title = {On general minimax theorems},
  journal = {Pacific J. Math.}, volume = {8}, number = {1}, pages = {171--176}, year = {1958}}

@inproceedings{McMahan2003,
  title={Planning in the presence of cost functions controlled by an adversary},
  author={McMahan, H Brendan and Gordon, Geoffrey J and Blum, Avrim},
  booktitle={Proc. 20th Int. Conf. Mach. Learn. ({ICML})},
  pages={536--543},
  year={2003}
}

@article{Feldbaum1960,
  title={Dual control theory. {I}},
  author={Feldbaum, Aleksandr Aronovich},
  journal={Avtomatika i Telemekhanika},
  volume={21},
  number={9},
  pages={1240--1249},
  year={1960}
}

@article{BarShalomTse1974,
  title={Dual effect, certainty equivalence, and separation in stochastic control},
  author={Bar-Shalom, Yaakov and Tse, Edison},
  journal={IEEE Trans. Autom. Control},
  volume={19},
  number={5},
  pages={494--500},
  year={1974},
  publisher={IEEE}
}

@article{HollingerSukhatme2014,
  title={Sampling-based robotic information gathering algorithms},
  author={Hollinger, Geoffrey A and Sukhatme, Gaurav S},
  journal={Int. J. Robot. Res.},
  volume={33},
  number={9},
  pages={1271--1287},
  year={2014},
  publisher={SAGE Publications Sage UK: London, England}
}

@book{Tambe2011,
  title={Security and game theory: algorithms, deployed systems, lessons learned},
  author={Tambe, Milind},
  year={2011},
  publisher={Cambridge university press}
}

@article{GilmoreGomory1961,
  title={A linear programming approach to the cutting-stock problem},
  author={Gilmore, Paul C and Gomory, Ralph E},
  journal={Operations research},
  volume={9},
  number={6},
  pages={849--859},
  year={1961},
  publisher={INFORMS}
}

@inproceedings{zhou2025integrating,
  title={Integrating Sequential Hypothesis Testing into Adversarial Games: A Sun Zi-Inspired Framework},
  author={Zhou, Haosheng and Ralston, Daniel and Yang, Xu and Hu, Ruimeng},
  booktitle={2025 IEEE 64th Conference on Decision and Control (CDC)},
  pages={4540--4546},
  year={2025}
}

@article{zhou2025adversarial,
  title={Adversarial Decision-Making in Partially Observable Multi-Agent Systems: A Sequential Hypothesis Testing Approach},
  author={Zhou, Haosheng and Ralston, Daniel and Yang, Xu and Hu, Ruimeng},
  journal={IEEE Trans. Control Netw. Syst.},
  year={2026}
}

@article{ralston2026information,
  title={Information Revelation and Alignment Faking in Stochastic Differential Games},
  author={Ralston, Daniel and Yang, Xu and Hu, Ruimeng},
  journal={arXiv:2603.17197},
  year={2026}
}

@article{kim2026deception,
  title={Deception in Linear-Quadratic Control},
  author={Kim, Yerin and Zhou, Haosheng and Benvenuti, Alexander and Hu, Ruimeng and Hale, Matthew},
  journal={arXiv:2604.00227},
  year={2026}
}

\end{document}